\documentclass[11pt]{article}
\usepackage[margin=0.72in]{geometry}
\usepackage{amsmath,amssymb,amsthm,microtype}
\usepackage{tikz}
\usepackage[hidelinks]{hyperref}
\newtheorem{theorem}{Theorem}
\newcommand{\R}{\mathbb R}
\title{Focal Sets of the John Ellipsoid of a Simplex from Its Vertices}
\author{Anatoly Eydelzon\\Department of Mathematical Sciences\\The University of Texas at Dallas\\\texttt{anatoly@utdallas.edu}}
\date{}
\begin{document}
\maketitle
\begin{abstract}
Starting with the vertices of a simplex in $\R^n$, we build a cubic vector field. If the semiaxes of the simplex's John ellipsoid are distinct, the Jacobian of this field has repeated eigenvalues exactly on the ellipsoid's focal sets. This works in every dimension $n\geq2$. We give examples in $\R^4$, $\R^3$, and $\R^2$: three focal quadrics, an ellipse and a hyperbola, and finally two points. In the plane, these two points are also the zeros of the derivative in the Siebeck--Marden theorem. Complex numbers explain why the planar case has a shorter formula.
\end{abstract}

\section*{The triangle and a question}
If the vertices of a triangle are $z_0,z_1,z_2\in\mathbb C$, the Siebeck--Marden theorem says that the zeros of
\[
\frac{d}{dz}\bigl[(z-z_0)(z-z_1)(z-z_2)\bigr]
\]
are the foci of the ellipse tangent to the sides at their midpoints. It is called the Steiner inellipse and is also the largest-area ellipse inside the triangle. For proofs of the Siebeck--Marden theorem, see \cite{Kalman,Bogosel}.

\begin{figure}[htbp]
\centering
\begin{tikzpicture}[scale=.58]
\draw[thin] (4,0)--(-2,{sqrt(3)})--(-2,{-sqrt(3)})--cycle;
\draw[green!45!black,very thick,samples=140,smooth,variable=\t,
  domain=0:360] plot ({2*cos(\t)},{sin(\t)});
\fill[orange!75!black] ({sqrt(3)},0) circle (2.5pt);
\fill[orange!75!black] ({-sqrt(3)},0) circle (2.5pt);
\end{tikzpicture}
\caption{A triangle, its Steiner inellipse, and its two foci. We return to these vertices after the higher-dimensional examples.}\label{fig:triangle}
\end{figure}
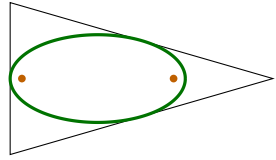

What happens to this picture for a simplex in $\R^3$, $\R^4$, or $\R^n$? Can we find the focal sets of its maximal-volume inscribed ellipsoid, called the \emph{John ellipsoid}, directly from the vertices? We cannot multiply vectors in $\R^3$ as we multiply complex numbers. We build a real vector field instead and look for repeated eigenvalues of its Jacobian. The method works in every dimension $n\geq2$ when the semiaxes are distinct. After the general calculation we look at $\R^4$, $\R^3$, and finally the triangle, where the complex polynomial reappears.

\section*{One matrix from the vertices}
Let $v_0,\ldots,v_n$ be the vertices of a nondegenerate simplex in $\R^n$, let $g=(v_0+\cdots+v_n)/(n+1)$, and put $u_i=v_i-g$. Form the positive-definite matrix
\begin{equation}\label{eq:B}
 B=\frac{1}{n(n+1)}\sum_{i=0}^{n}u_i u_i^T.
\end{equation}
Its John ellipsoid has the equation
\begin{equation}\label{eq:E}
 E=\bigl\{g+x:x^TB^{-1}x\leq 1\bigr\}.
\end{equation}
To see this, start with a regular simplex of circumradius $R$. Its inscribed ball has radius $R/n$ and is its John ellipsoid. Symmetry gives $\sum u_i u_i^T=(n+1)R^2 I/n$, so \eqref{eq:B} gives the ball's matrix $R^2I/n^2$. Now send this regular simplex onto our simplex by an invertible affine map $v\mapsto Av+b$. Its centered vertices become $Au_i$, so the matrix of our simplex in \eqref{eq:B} is $A(R^2I/n^2)A^T=(R^2/n^2)AA^T$. We call this matrix $B$. The inscribed ball becomes the ellipsoid \eqref{eq:E}. It must still have the greatest volume: every competing ellipsoid pulls back to one in the regular simplex, and all volumes change by the same factor $|\det A|$. Thus this affine image is both the ellipsoid in \eqref{eq:E} and the John ellipsoid; see also \cite{John,Ball}. The Euclidean inscribed ball of our simplex may be different. A corner-volume characterization of the same John ellipsoid appears in \cite{Eydelzon2026}. Our question here is different: how can we find its focal sets from the vertices?

The matrix $B$ is symmetric, so an orthogonal matrix $Q$ diagonalizes it:
\[
 Q^TBQ=D=\operatorname{diag}(d_1,\ldots,d_n),\qquad d_1>\cdots>d_n>0.
\]
We place the origin at $g$ and use coordinates along the columns of $Q$. In these coordinates the matrix of the ellipsoid is $D$. To keep the notation short, we write $x$ for the new coordinate vector and $B=D$ from this point on. An orthogonal change of coordinates does not affect whether a Jacobian has a repeated eigenvalue. Here $d_i$ is the $i$th eigenvalue of the original $B$, so the semiaxis lengths are $a_i=\sqrt{d_i}$. The strict inequalities are needed: equal semiaxes can make the Jacobian have repeated eigenvalues away from the foci. For example, if $D=\operatorname{diag}(9,4,4)$, then at every point $(t,0,0)$ the matrix $D-xx^T=\operatorname{diag}(9-t^2,4,4)$ has a repeated eigenvalue. It detects the entire first axis, while the ellipsoid $x_1^2/9+(x_2^2+x_3^2)/4=1$ has just two foci, $(\pm\sqrt5,0,0)$. Thus our exact correspondence requires distinct semiaxes.

\section*{A cubic and its Jacobian}
Consider the cubic vector field
\begin{equation}\label{eq:F}
 F(x)=Bx-\tfrac12\|x\|^2x.
\end{equation}
Its Jacobian is the symmetric matrix
\begin{equation}\label{eq:J}
 DF(x)=B-xx^T-\tfrac12\|x\|^2I.
\end{equation}
The last term shifts every eigenvalue by the same amount. Therefore $DF(x)$ has a repeated eigenvalue if and only if $D-xx^T$ does.

We should say what we mean by a \emph{focal quadric} in $\R^n$. The ordinary members of the confocal family are
\[
 \sum_{j=1}^n\frac{x_j^2}{d_j-t}=1,\qquad t\ne d_1,\ldots,d_n.
\]
At the singular parameter $t=d_k$, we call the set
\[
 x_k=0,\qquad \sum_{j\ne k}\frac{x_j^2}{d_j-d_k}=1
\]
the associated focal quadric, following \cite[Eq.~(15)]{HorvathProk}. One cannot simply substitute $t=d_k$ into the confocal equation. In particular, a limiting surface can flatten over a region of $x_k=0$; the displayed equation selects its focal boundary. Each nonempty focal quadric has dimension $n-2$ as a set in $\R^n$. For $n=3$ these sets are the focal curves; for $n=2$ they are the two foci. We use ``focal set'' for all these cases.

\begin{theorem}\label{thm:main}
For any $n\geq2$, let a simplex in $\R^n$ have John ellipsoid \eqref{eq:E} with distinct squared semiaxes $d_1>\cdots>d_n$. The points where $DF(x)$ has a repeated eigenvalue form the union of the following $n-1$ real sets:
\begin{equation}\label{eq:focal}
 x_k=0,\qquad \sum_{j\ne k}\frac{x_j^2}{d_j-d_k}=1,
 \qquad k=2,\ldots,n.
\end{equation}
These are exactly its real focal sets in principal coordinates centered at $g$.
\end{theorem}

\begin{proof}
Write $M=D-xx^T$. If $M$ has a repeated eigenvalue $\lambda$, its eigenspace has dimension at least two. We can therefore choose a nonzero eigenvector $y$ perpendicular to $x$. Then
\[
 (D-\lambda I)y=xx^Ty=0.
\]
Since the $d_j$ are distinct, $\lambda=d_k$ for some $k$, and $y$ is a multiple of the $k$th coordinate vector $e_k$. In particular, $x_k=0$. Conversely, when $x_k=0$, the vector $e_k$ is an eigenvector of $M$ with eigenvalue $d_k$. To have a repeated eigenvalue, the remaining $(n-1)\times(n-1)$ block must also have eigenvalue $d_k$. We use $\det(S-ww^T)=\det(S)(1-w^TS^{-1}w)$ for an invertible matrix $S$. Here $S$ is diagonal with entries $d_j-d_k$ for $j\ne k$, and $w$ consists of the coordinates $x_j$ with $j\ne k$. Thus the determinant of the remaining block minus $d_k I$ is
\[
 \prod_{j\ne k}(d_j-d_k)
 \left(1-\sum_{j\ne k}\frac{x_j^2}{d_j-d_k}\right).
\]
The determinant is zero exactly on \eqref{eq:focal}. The case $k=1$ gives no real points, because all the denominators are negative. For the other values of $k$, these are precisely the focal equations defined above; see \cite{Tabachnikov,Dragovic} for the confocal construction.
\end{proof}

We get one equation for each $k=2,\ldots,n$, in every dimension. No complex numbers are needed. What do these sets look like? Each lies in the hyperplane $x_k=0$ and has dimension $n-2$. Its equation has $k-1$ positive square terms and $n-k$ negative ones. For $k=n$, it is an ellipsoid in the hyperplane $x_n=0$. For $n\geq3$, the other values of $k$ give hyperboloids. When $n=2$, only two points remain. When $n=3$, we get an ellipse and a hyperbola. Our next three examples show how this pattern works.

The focal quadrics do not determine the whole ellipsoid. If we replace $D$ by $D+tI$, the differences $d_j-d_k$ stay the same, so \eqref{eq:focal} does not change. The Jacobian also changes only by $tI$. The vertices, through \eqref{eq:B}, tell us which ellipsoid we started with. Even in the plane, knowing the two foci is not enough to know the ellipse.

The focal-quadric equations themselves are classical \cite{HorvathProk,Dragovic}, and rank-one changes of diagonal matrices have an established eigenvalue theory \cite{Golub}. Our point is that the repeated-eigenvalue locus of the explicit cubic Jacobian \eqref{eq:J} recovers all these focal quadrics from the simplex vertices and, in the plane, becomes the derivative in the Siebeck--Marden theorem.

\section*{A four-dimensional example}
We first take a $4$-simplex. Its five vertices are the columns of the matrix
\[
 U=\sqrt{10}\begin{pmatrix}
 1&1&1&1&-4\\
 1&1&1&-3&0\\
 1&1&-2&0&0\\
 1&-1&0&0&0
 \end{pmatrix}.
\]
Each row sums to zero, so the centroid is the origin. The rows are also orthogonal. By \eqref{eq:B}, with $n=4$,
\[
 B=\frac1{20}UU^T=\operatorname{diag}(10,6,3,1),
 \qquad E:\ \frac{x_1^2}{10}+\frac{x_2^2}{6}+\frac{x_3^2}{3}+x_4^2\leq1.
\]
The theorem gives three focal quadrics. We list them from the smallest eigenvalue upward:
\begin{align*}
 k=4:&\quad x_4=0, &\frac{x_1^2}{9}+\frac{x_2^2}{5}+\frac{x_3^2}{2}&=1,\\
 k=3:&\quad x_3=0, &\frac{x_1^2}{7}+\frac{x_2^2}{3}-\frac{x_4^2}{2}&=1,\\
 k=2:&\quad x_2=0, &\frac{x_1^2}{4}-\frac{x_3^2}{3}-\frac{x_4^2}{5}&=1.
\end{align*}
The first is an ellipsoid, the second a one-sheet hyperboloid, and the third a two-sheet hyperboloid. Each is a two-dimensional surface in its indicated three-dimensional hyperplane. So the five vertices give us three kinds of focal quadric in $\R^4$.

\section*{A tetrahedron gives two curves}
Take the four vertices
\[
\begin{split}
 v_0&=(2\sqrt3,\sqrt6,\sqrt3),\quad v_1=(2\sqrt3,-\sqrt6,-\sqrt3),\\
 v_2&=(-2\sqrt3,\sqrt6,-\sqrt3),\quad v_3=(-2\sqrt3,-\sqrt6,\sqrt3).
\end{split}
\]
Their centroid is zero. When we add the matrices $v_iv_i^T$, the off-diagonal entries cancel. From \eqref{eq:B},
\[
 B=\frac1{12}\sum_{i=0}^{3}v_iv_i^T
 =\operatorname{diag}(4,2,1),\qquad
 E:\ \frac{x^2}{4}+\frac{y^2}{2}+z^2\leq1.
\]
For $k=3$, \eqref{eq:focal} gives the focal ellipse
\begin{equation}\label{eq:ell}
 z=0,\qquad \frac{x^2}{3}+y^2=1.
\end{equation}
For $k=2$, it gives the focal hyperbola
\begin{equation}\label{eq:hyp}
 y=0,\qquad \frac{x^2}{2}-z^2=1.
\end{equation}
So from the four vertices we obtain both focal curves. They are shown in Figure~\ref{fig:tetra}.

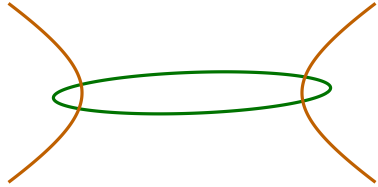
\begin{figure}[htbp]
\centering
\begin{tikzpicture}[scale=1.03,
  x={(1cm,0cm)},y={(.42cm,.27cm)},z={(0cm,.86cm)}]
\draw[green!45!black,very thick,samples=140,smooth,variable=\t,
  domain=0:360] plot ({sqrt(3)*cos(\t)},{sin(\t)},0);
\draw[orange!75!black,very thick,samples=90,smooth,variable=\t,
  domain=-1.1:1.1] plot ({sqrt(2)*cosh(\t)},0,{sinh(\t)});
\draw[orange!75!black,very thick,samples=90,smooth,variable=\t,
  domain=-1.1:1.1] plot ({-sqrt(2)*cosh(\t)},0,{sinh(\t)});
\end{tikzpicture}
\caption{The tetrahedron's focal ellipse (green) in $z=0$ and focal hyperbola (orange) in $y=0$.}\label{fig:tetra}
\end{figure}

\section*{Back to the triangle}
Consider the vertices $(4,0)$, $(-2,\sqrt3)$, $(-2,-\sqrt3)$. Their centroid is zero, and \eqref{eq:B} gives $B=\operatorname{diag}(4,1)$. The Steiner inellipse is therefore $x^2/4+y^2\leq1$. The only focal set in \eqref{eq:focal} is
\[
 y=0,\qquad \frac{x^2}{4-1}=1,
\]
namely $(\pm\sqrt3,0)$; see Figure~\ref{fig:triangle}. If we now view the vertices as complex numbers, their polynomial is
\[
 p(z)=(z-4)(z+2-i\sqrt3)(z+2+i\sqrt3)=z^3-9z-28,
\]
and $p'(z)=3(z^2-3)$ has zeros $\pm\sqrt3$. We can also check the Jacobian directly. Its two diagonal entries agree when $3-x^2+y^2=0$, and its off-diagonal entry vanishes when $xy=0$. Together these give $y=0$ and $x=\pm\sqrt3$.

\section*{Why the plane has a simpler formula}
For an arbitrary centered triangle, write the vertices as complex numbers $z_0,z_1,z_2$, with $z_0+z_1+z_2=0$, and put $p(z)=\prod_{i=0}^2(z-z_i)$. Since
\[
 z_0z_1+z_0z_2+z_1z_2=-\frac12\sum_{i=0}^2z_i^2,
\]
we have
\begin{equation}\label{eq:pprime}
 \frac{p'(z)}3=z^2-\frac16\sum_{i=0}^2z_i^2.
\end{equation}
Now write $z_i=s_i+it_i$. Formula~\eqref{eq:B}, with $n=2$, gives
\begin{equation}\label{eq:identity}
 \frac16\sum_{i=0}^2z_i^2=(B_{11}-B_{22})+2iB_{12}.
\end{equation}
Why put the entries of $B$ together in this way? A real symmetric $2\times2$ matrix $A$ has a repeated eigenvalue exactly when $A_{11}=A_{22}$ and $A_{12}=0$. In complex notation these two equations become one:
\[
 (A_{11}-A_{22})+2iA_{12}=0.
\]
Use this with $A=DF(x,y)$ and $z=x+iy$. By \eqref{eq:J}, the term $-\tfrac12(x^2+y^2)I$ disappears when we subtract the diagonal entries, and
\begin{align*}
 (DF_{11}-DF_{22})+2iDF_{12}
 &= (B_{11}-B_{22})+2iB_{12}-(x^2-y^2+2ixy)\\
 &=\frac16\sum_{i=0}^2z_i^2-z^2
 =-\frac{p'(z)}3.
\end{align*}
So, for every triangle, $p'(z)=0$ exactly when $DF$ has a repeated eigenvalue. If the triangle is not centered at the origin, we first subtract its centroid and then add it back to the two resulting points.

The plane is special: we can identify $\R^2$ with the field $\mathbb C$. We can multiply the factors $(z-z_i)$ to make a scalar polynomial and then take its ordinary derivative. Also $z^2=(x^2-y^2)+2ixy$ puts the two real equations for a repeated eigenvalue into one equation. Under a rotation by $\theta$, the complex number in \eqref{eq:identity} picks up a factor $e^{2i\theta}$; this gives the direction of the line through the foci. There is no comparable complex field structure on $\R^3$ or $\R^4$ that would give such a vertex polynomial. The real matrix $B$ and the Jacobian work in all dimensions. The complex derivative is their simpler planar form.

\end{document}